\documentclass[reqno,11pt]{amsart}

\usepackage{amsmath,amssymb,amsthm,amsfonts,mathrsfs,latexsym,mathtools,bm,tikz}
\usepackage[abbrev]{amsrefs}
\usepackage{stmaryrd} 
\usepackage[OT1, T1]{fontenc}
\usepackage{cleveref}
\usepackage{autonum}
\usepackage[marginparwidth=0pt,margin=25truemm]{geometry}

\usepackage{ytableau} 
\usepackage{braket}
\ytableausetup{boxsize=1.4em}
\ytableausetup{centertableaux}

\theoremstyle{plain}
\newtheorem{thm}{Theorem}[section]
\newtheorem{lem}[thm]{Lemma}

\theoremstyle{definition}

\newtheorem{dfn}[thm]{Definition}
\newtheorem{rem}[thm]{Remark}

\usetikzlibrary{intersections,calc,arrows.meta,decorations,decorations.pathreplacing}

\allowdisplaybreaks[2]

\everymath{\displaystyle}

\newcommand{\dep}{\mathrm{dep}}
\newcommand{\wt}{\mathrm{wt}}
\newcommand{\inle}{\vartriangleleft}
\newcommand{\inleq}{\trianglelefteq}

\newcommand{\ZZ}{\mathbb{Z}}
\newcommand{\RR}{\mathbb{R}}

\title[A $q$-analogue of Maesaka--Seki--Watanabe's formula]%
{Maesaka--Seki--Watanabe's formula for multiple harmonic $q$-sums}
\author[Y.~TSURUTA]{YUTO TSURUTA}
\date{}
\address{Mathematical Institute, Tohoku University Sendai 980-8578 Japan}
\email{tsuruta.yuuto.q7@dc.tohoku.ac.jp}

\begin{document}
\begin{abstract}
Maesaka, Seki, and Watanabe recently discovered an equality called the MSW formula. This paper provides a $q$-analogue of MSW formula. It discusses the new proof of the duality relation for finite multiple harmonic $q$-series at primitive roots of unity via $q$-analogue of MSW formula. This paper also gives a $q$-analogue of Yamamoto's generalization of MSW formula for Schur type.

\end{abstract}
\maketitle

\section{Introduction}
A tuple of positive integers $\bm{k}=(k_{1},\ldots,k_{r})$ is called an \textit{index}. For any index $\bm{k}=(k_{1},\ldots,k_{r})$, we call $\wt(\bm{k})=k_{1}+\cdots+k_{r}$ its \textit{weight} and $\dep(\bm{k})=r$ its \textit{depth}. Then the \textit{multiple harmonic sum} $\zeta_{<N}(\bm{k})$ of non-empty index $\bm{k}$ is defined by
  $$\zeta_{<N}(\bm{k})=\sum_{0<m_{1}<\cdots<m_{r}<N}\frac{1}{m_{1}^{k_{1}}\cdots m_{r}^{k_{r}}},$$
and $\zeta_{<N}(\varnothing)=1$, where $\varnothing$ is the \textit{empty index}. If $k_{r}\geq2$, the \textit{multiple zeta value} of index $\bm{k}$ is
  \begin{align}
  \zeta({\bm{k}})=\lim_{N\to\infty}\zeta_{<N}(\bm{k})=\sum_{0<m_{1}<\cdots<m_{r}}\frac{1}{m_{1}^{k_{1}}\cdots m_{r}^{k_{r}}}=\int_{0}^{1}\dfrac{\text{d}t}{1-t}\Big(\dfrac{\text{d}t}{t}\Big)^{k_1-1}\cdots\dfrac{\text{d}t}{1-t}\Big(\dfrac{\text{d}t}{t}\Big)^{k_r-1}\in\RR.
    \end{align}
Recently, Maesaka, Seki and Watanabe (\cite{MSWoriginal}) found the innovative finite sum, which is called the \textit{$\flat$-sum} of multiple harmonic sum
\begin{align}
\label{flat1}
\zeta_{<N}^{\flat}(\bm{k})&=\sum_{\substack{0<n_{j1}\leq\cdots\leq n_{jk_{j}}<N(1\leq j\leq r)\\n_{jk_{j}}<n_{(j+1)1}(1\leq j<r) }}\prod_{j=1}^{r}\frac{1}{(N-n_{j1})n_{j2}\cdots n_{jk_{j}}},
\end{align}
and discovered an interesting result called the \textit{MSW formula} (Maesaka--Seki--Watanabe's formula). In the proof of MSW formula, they use the \textit{connected sum method}, which is established in \cite{SY1} (see more details in \cite{CSMseki}).
\begin{thm}[MSW formula, {\cite[Theorem 1.3]{MSWoriginal}}]\label{MSWformula}
For any $N>0$ and $\bm{k}=(k_{1},\ldots,k_{r})\in(\ZZ_{>0})^r$, it holds
$$\zeta_{<N}(\bm{k})=\zeta_{<N}^{\flat}(\bm{k}).$$
\end{thm}

MSW formula gives a \textit{discretization} of iterated integral expression of the multiple zeta value, in a sense. The duality and the extended double shuffle relations for multiple zeta values are 
reobtained by using MSW formula (see \cite{MSWoriginal}, \cite{MSWEDSR}). Furthermore, Kawamura (\cite{MSWcyclic}) established the framework of the \textit{discrete iterated integrals}, and also obtained the cyclic sum formula for multiple polylogarithms. Hirose--Matsusaka--Seki (\cite{MSWpolylog}) executed a discretization of iterated integral expression of the multiple polylogarithm. As we can see from these examples, MSW formula is expected to have many applications and generalizations.

The present article aims to extend MSW formula to its $q$-analogue. Generally, $q$-integer $[n]_{q},\ n\in\ZZ_{\geq0}$ is defined by $[n]_{q}=(1-q^n)/(1-q)$ and $[0]_{q}=0$. Note that $[n]_{q}\to n$ as $q\to1$ for any $n>1$. In \S2, we give a $q$-analogue of MSW formula and related results:

\begin{thm}\label{main thm 1}
  For any $N>0$ and index $\bm{k}=(k_{1},\ldots,k_{r})$, we have
  $$\zeta_{<N}^{BZ}(\bm{k};q)=\zeta_{<N}^{q\flat}(\bm{k}),$$
  where
  \begin{align}
    \zeta_{<N}^{q\flat}(\bm{k})&=\sum_{\substack{0<n_{j1}\leq\cdots\leq n_{jk_{j}}<N(1\leq j\leq r)\\ n_{jk_{j}}<n_{(j+1)1}(1\leq j<r)}}\prod_{j=1}^{r}\frac{q^{n_{j2}+\cdots+n_{jk_{j}}}}{[N-n_{j1}]_{q}[n_{j2}]_{q}\cdots[n_{jk_{j}}]_{q}},\\
    \zeta_{<N}^{BZ}(\bm{k};q)&=\sum_{0<m_{1}<\cdots<m_{r}<N}\frac{q^{(k_{1}-1)m_{1}+\cdots+(k_{r}-1)m_{r}}}{[m_{1}]_{q}^{k_{1}}\cdots[m_{r}]_{q}^{k_{r}}}.
  \end{align}
\end{thm}
Here, $\zeta_{<N}^{BZ}(\bm{k})=\zeta_{<N}^{BZ}(\bm{k};q)$ is called the multiple harmonic $q$-sum of \textit{Bradley--Zhao model} (cf. \cite{Bradley1}). If $k_{r}\geq2$, $\zeta_{<N}^{BZ}(\bm{k})$ converges as $N\to\infty$ and $\lim_{q\to1}\lim_{N\to\infty}\zeta_{<N}^{BZ}(\bm{k})=\zeta(\bm{k})$. In this paper, we also give a new proof of the following:
  \begin{thm}[K. Hessami Pilehrood--T. Hessami Pilehrood--Tauraso {\cite[Theorem 3.1]{KhHessamiPilehrood1}}]\label{Thm:Hessami}
   Let $N$ be a positive integer and $\zeta_{N}$ be a primitive $N$-th root of unity. For any index $\bm{k}=(k_{1},\ldots,k_{r})$, it holds
    $$\zeta_{<N}^{BZ}(\bm{k};\zeta_{N})=(-1)^{r}\sum_{\substack{0<n_{j1}\leq\cdots\leq n_{jk_{j}}<N(1\leq j\leq r)\\ n_{jk_{j}}<n_{(j+1)1}(1\leq j<r)}}\prod_{j=1}^{r}\frac{\zeta_{N}^{n_{j1}+\cdots+n_{jk_{j}}}}{[n_{j1}]_{\zeta_{N}}\cdots[n_{jk_{j}}]_{\zeta_{N}}}.$$
\end{thm}
The formula in \Cref{Thm:Hessami} is proved in \cite{KhHessamiPilehrood1} by using Theorem A (cf. {\cite[Theorem 8.1]{KhHessamiPilehrood2}}), in the research based on the study in Bachmann--Takeyama--Tasaka \cite{BTT1,BTT2}. In the present paper, we give the proof of \Cref{Thm:Hessami} by applying \Cref{main thm 1}. Since Theorem A and \Cref{main thm 1} are essentially different identities of each other, we can assure that it is a new proof of \Cref{Thm:Hessami}. Furthermore, we regard \Cref{Thm:Hessami} as a $q$-analogue of the Hoffman duality ({\cite[Theorem 4.6]{Hoffman1}}). So we can affirm that the proof of \Cref{Thm:Hessami} provides a $q$-analogue of the discussion of section 3 in \cite{MSWoriginal}. 

Following the study in Nakasuji--Phuksuwan--Yamasaki \cite{SchurNPY}, Yamamoto (\cite{MSWyamamoto}) extended MSW formula to (skew-)Schur multiple harmonic sums, defined by
$$\zeta_{<N}(\bm{k})=\sum_{(m_{ij})\in\text{SSYT}_{<N}(D)}\prod_{(i,j)\in D}\frac{1}{m_{ij}^{k_{ij}}}.$$
Here, $\bm{k}=(k_{ij})$ is an index on a skew Young diagram $D$ , i.e. $(i,j)\in D$ (cf. \cite{HMOSchurInt}), and $\text{SSYT}_{<N}(D)$ is the set of semi-standard Young tableaux on $D$ whose entries are positive integers less than $N$. Yamamoto defined the Schur type of the $\flat$-sum $\zeta_{<N}^{\flat}(\bm{k})$. Hereinafter we abbreviate $\flat$-sum as $\zeta_{<N}^{\flat}(\bm{k})$ regardless of the shape of index $\bm{k}$. Note that the following theorem is a generalization of \Cref{MSWformula}.

\begin{thm}[Yamamoto {\cite[Theorem 3.6]{MSWyamamoto}}]\label{Thm:Yam3.6}
  For any diagonally constant index $\bm{k}=(k_{ij})$ (i.e. $k_{ij}$ depends only on $i-j$) and $N>0$, it holds
  $$\zeta_{<N}(\bm{k})=\zeta_{<N}^{\flat}(\bm{k}).$$
\end{thm}

In \S3, we prove a $q$-analogue of \Cref{Thm:Yam3.6} which is stated as follows:

\begin{thm}[main result]\label{main thm 2}
  For any $N>0$ and diagonally constant index $\bm{k}$, we have
  $$\zeta_{<N}^{BZ}(\bm{k})=\zeta_{<N}^{q\flat}(\bm{k}),$$
  where $\zeta_{<N}^{q\flat}(\bm{k})$ is the $q$-analogue of $\flat$-sum $\zeta_{<N}^{\flat}(\bm{k})$.
\end{thm}
Note that \Cref{main thm 2} includes \Cref{main thm 1} as the case when $\bm{k}=\begin{ytableau}
k_{1} \\ \vdots \\  k_{r}
\end{ytableau}$ and \Cref{Thm:Yam3.6}.

\section{Extension of the MSW formula} 
First, we begin to state some notations. Let $N$ be a positive integer. For non-negative integers $m,\ n$ satisfying $m\leq n\leq N$, we define the \textit{connector} $C_{N}(m,n)$ as
$$C_{N}(m,n)=\left. \binom{n}{m} \middle/ \binom{N}{m}. \right. $$
In addition, for $\bm{k}=(k_{1},\ldots,k_{r})$ and $\bm{l}=(l_{1},\ldots,l_{s})$, we define the \textit{connected sum} $Z_{N}(\bm{k}\ |\ \bm{l})$ as
$$Z_{N}(\bm{k}\ |\ \bm{l})=\sum_{\substack{0<m_{1}<\cdots<m_{r}\leq n_{11}\\ n_{j1}\leq\cdots\leq n_{jl_{j}}\leq n_{(j+1)1}(1\leq j<s)\\ n_{(j+1)1}<n_{(j+1)2}(1\leq j<s)\\ n_{s2}\leq\cdots\leq n_{sl_{s}}<N}}\frac{1}{m_{1}^{k_{1}}\cdots m_{r}^{k_{r}}}\cdot C_{N}(m_{r},n_{11})\cdot\prod_{j=1}^{s}\frac{1}{(N-n_{j1})n_{j2}\cdots n_{jl_{j}}}.$$

\begin{rem}
  By definition, if $\bm{k}=\varnothing$ or $\bm{l}=\varnothing$, we set
  $$Z_{N}(\bm{k}\ |)=\zeta_{<N+1}(\bm{k}),\ Z_{N}(|\ \bm{l})=\zeta_{<N+1}^{\flat}(\bm{k}),$$
 respectively. We often call it as the \textit{boundary conditions}. 
\end{rem}
The proof of \Cref{MSWformula} uses the following lemmas. 

\begin{lem}[{\cite[Lemma 4.1]{MSWoriginal}}]\label{Lem:MSW4.1}
  For non-negative integers $m,n$ satisfying $m\leq n\leq N$, it holds
\begin{align}
   \frac{1}{m}C_{N}(m,n)&=\sum_{m\leq b\leq n}C_{N}(m,b)\frac{1}{b}\quad (m>0), \\
   \sum_{m<a\leq n}\frac{1}{a}C_{N}(a,n)&=\sum_{m\leq b<n}C_{N}(m,b)\frac{1}{N-b}\quad (m<n).
\end{align}
\end{lem}

\begin{lem}[{\cite[Lemma 4.2]{MSWoriginal}}]\label{Lem:MSW4.2}
  For $k\in\ZZ_{>0}$ and indices $\bm{k},\ \bm{l}$, it holds
  \begin{alignat}{3}
    Z_{N}(\bm{k},k\ |\ \bm{l})&=Z_{N}(\bm{k}\ |\ k,\bm{l}),\\
    Z_{N}(\bm{k},k\ |)&=Z_{N}(\bm{k}\ |\ k),\\
    Z_{N}(k\ |\ \bm{l})&=Z_{N}(|\ k,\bm{l}).
  \end{alignat}
\end{lem}

We often call the relations in \Cref{Lem:MSW4.2} as the \textit{transport relations}. By transport relations, \Cref{MSWformula} holds as follows:
$$\zeta_{<N+1}(\bm{k})=Z_{N}(\bm{k}\ |)=\cdots=Z_{N}(k_{1},\ldots,k_{r-1}\ |\ k_{r})=\cdots=Z_{N}(|\ \bm{k})=\zeta_{<N+1}^{\flat}(\bm{k}).$$
In the next subsection, we aim to extend above \Cref{Lem:MSW4.1} and \Cref{Lem:MSW4.2} to their $q$-analogue.

\subsection{Proof of \Cref{main thm 1}}
For $m,n\in\ZZ_{\geq0}$ with $m\leq n$, we define
$$\binom{n}{m}_{q}=\frac{[n]_{q}!}{[m]_{q}![n-m]_{q}!},\quad[n]_{q}!=\prod_{j=1}^n[j]_{q},\quad [0]_{q}!=1.$$
\begin{lem}
  For non-negative integers $m,n$ satisfying $m\leq n\leq N$, we have
  \begin{align}
    \label{1}\frac{q^m}{[m]_{q}}C^{q}_{N}(m,n)&=\sum_{m\leq b\leq n}C^{q}_{N}(m,b)\frac{q^b}{[b]_{q}}\quad (m>0),\\
    \label{2}\sum_{m<a\leq n}\frac{1}{[a]_{q}}C^{q}_{N}(a,n)&=\sum_{m\leq b<n}C^{q}_{N}(m,b)\frac{1}{[N-b]_{q}}\quad (m<n),
  \end{align}
      where
    $$C^{q}_{N}(m,n)=\left. \binom{n}{m}_{q} \middle/ \binom{N}{m}_{q}. \right.$$
\end{lem}
  
\begin{proof}
  It is sufficient to prove
 $$\frac{q^m}{[m]_{q}}\Big(C^{q}_{N}(m,b)-C^{q}_{N}(m,b-1)\Big)=C^{q}_{N}(m,b)\frac{q^b}{[b]_{q}},\quad 0<m<b\leq n,$$
  and
  $$\frac{1}{[a]_{q}}\Big(C^{q}_{N}(a,b+1)-C^{q}_{N}(a,b)\Big)=\Big(C^{q}_{N}(a-1,b)-C^{q}_{N}(a,b)\Big)\frac{1}{[N-b]_{q}},\quad m<a\leq b<n.$$
  First, for $0<m<b\leq n$, we have
  \begin{alignat}{3}
    \frac{q^m}{[m]_{q}}\Big(C^{q}_{N}(m,b)-C^{q}_{N}(m,b-1)\Big)&=\frac{q^m}{[m]_{q}}\Bigg(\frac{[b]_{q}!}{[m]_{q}![b-m]_{q}!}\cdot\binom{N}{m}_{q}^{-1}-\frac{[b-1]_{q}!}{[m]_{q}![b-1-m]_{q}!}\cdot\binom{N}{m}_{q}^{-1}\Bigg)\\
    &=\frac{q^m}{[m]_{q}}\binom{b}{m}_{q}\binom{N}{m}_{q}^{-1}\Bigg(1-\frac{[b-m]_{q}}{[b]_{q}}\Bigg)\\
    &=C_{N}^{q}(m,b)\frac{q^b}{[b]_{q}}.
  \end{alignat}
  Second, for $m<a\leq b<n$ we have
  \begin{alignat}{2}
    \frac{1}{[a]_{q}}\Big(C^{q}_{N}(a,b+1)-C^{q}_{N}(a,b)\Big)&=C_{N}^{q}(a,b)\frac{q^{b-a+1}}{[b-a+1]_{q}},\\
    \Big(C^{q}_{N}(a-1,b)-C^{q}_{N}(a,b)\Big)\frac{1}{[N-b]_{q}}&=C_{N}^{q}(a,b)\frac{q^{b-a+1}}{[b-a+1]_{q}}.
  \end{alignat}
\end{proof}
Next, for $\bm{k}=(k_{1},\ldots,k_{r})$ and $\bm{l}=(l_{1},\ldots,l_{s})$, we define $q$-analogue of $Z_{N}(\bm{k}\ |\ \bm{l})$ as
  \begin{alignat}{2}
    Z^{q}_{N}(\bm{k}\ |\ \bm{l})=\sum_{\substack{0<m_{1}<\cdots<m_{r}\leq n_{11}\\ n_{j1}\leq\cdots\leq n_{jl_{j}}\leq n_{(j+1)1}(1\leq j<s)\\ n_{(j+1)1}<n_{(j+1)2}(1\leq j<s)\\ n_{s2}\leq\cdots\leq n_{sl_{s}}<N}}&\frac{q^{(k_{1}-1)m_{1}+\cdots+(k_{r}-1)m_{r}}}{[m_{1}]_{q}^{k_{1}}\cdots [m_{r}]_{q}^{k_{r}}}\\
    &\times C^{q}_{N}(m_{r},n_{11})\prod_{j=1}^{s}\frac{q^{n_{j2}+\cdots+n_{jl_{j}}}}{[N-n_{j1}]_{q}[n_{j2}]_{q}\cdots [n_{jl_{j}}]_{q}},
  \end{alignat}
and the boundary condition is
$$Z_{N}^{q}(\bm{k}\ |)=\zeta_{<N+1}^{BZ}(\bm{k}),\ Z_{N}^{q}(|\ \bm{l})=\zeta_{<N+1}^{q\flat}(\bm{l}).$$
\begin{lem}[transport relations]\label{my lem 1}
 For any $k\in\ZZ_{>0}$ and indices $\bm{k},\ \bm{l}$, we have
  \begin{alignat}{3}
    Z^{q}_{N}(\bm{k},k\ |\ \bm{l})&=Z^{q}_{N}(\bm{k}\ |\ k,\bm{l}),\\
    Z^{q}_{N}(\bm{k},k\ |)&=Z^{q}_{N}(\bm{k}\ |\ k),\\
    Z^{q}_{N}(k\ |\ \bm{l})&=Z^{q}_{N}(|\ k,\bm{l}).
  \end{alignat}
       
\end{lem}
\begin{proof}
 By using \eqref{1} $k-1$ times and \eqref{2}, it holds from the definition of $Z_{N}^{q}(\bm{k}\ |\ \bm{l})$.
\end{proof}

\begin{proof}[Proof of \Cref{main thm 1}]
  By \Cref{my lem 1}, we have
  \begin{alignat}{4}
    \zeta_{<N+1}^{BZ}(\bm{k})&=Z_{N}^{q}(k_{1},\ldots,k_{r}\ |)\\
    &=Z_{N}^{q}(k_{1},\ldots,k_{r-1}\ |\ k_{r})\\
    &\vdots\\
    &=Z_{N}^{q}(|\ k_{1},\ldots,k_{r})=\zeta_{<N+1}^{q\flat}(\bm{k}).
  \end{alignat}
This gives a proof.
\end{proof}

\begin{proof}[Proof of \Cref{Thm:Hessami}]
  Since $\zeta_{N}$ satisfies $\zeta_{N}^{N}=1$, we obtain
  $$\dfrac{1}{[N-n_{j1}]_{\zeta_{N}}}=-\frac{\zeta_{N}^{n_{j1}}}{[n_{j1}]_{\zeta_{N}}},\quad j=1,\ldots,r.$$
  Hence we have the desired conclusion by \Cref{main thm 1}.
\end{proof}

\section{Extension of Yamamoto's work}
\subsection{Notation}
In this subsection, we review some notations in Yamamoto \cite{MSWyamamoto}. For a diagonally constant index $\bm{k}$ on Young diagram $D$, we set
\begin{alignat}{3}
  D_{p}&=\{(i,j)\in D\ |\ i-j=p\},\\
  p_{0}&=\min\{p\ |\ D_{p}\neq\emptyset\},\ p_{1}=\max\{p\ |\ D_{p}\neq\emptyset\},\\
  k_{p}&=k_{ij}\ \text{for}\ p_{0}\leq p\leq p_{1}\ \text{and}\ (i,j)\in D,\\
  [a,b]&=\{x\in\ZZ\ |\ a\leq x\leq b\},\ a,b\in\ZZ.
\end{alignat}
Note that, for $p\in\{p_{0},\ldots,p_{1}\}$, there is a bijection
$$D_{p}\longrightarrow J_{p}=\{j\in\ZZ\ |\ (j+p,j)\in D_{p}\}\quad (i,j)\mapsto j.$$
Let $(J,J^\prime)$ be a pair of intervals in $\ZZ$. If $J=[j_{0},j_{1}]$, we say a pair $(J,J^\prime)$ is \textit{consecutive} if $J^\prime=J$ or $J\setminus\{j_{1}\}$ or $J\cup\{j_{0}-1\}$ or $J\setminus\{j_{1}\}\cup\{j_{0}-1\}$. For $\bm{m}=(m_{j})\in[1,N-1]^{J}$ with $j\in J$, we define $\prod(\bm{m})=\prod_{j\in J}m_{j}$. 
\begin{dfn}[cf. \cite{MSWyamamoto}]
  Let $(J,J^\prime)$ be a consecutive pair of intervals and $\bm{m}=(m_{j})_{j\in J},\ \bm{n}=(n_{j})_{j\in J^\prime}$ be tuples of integers. Then we define 
$$\bm{m}\inle\bm{n}\xLeftrightarrow[]{\text{def}}\begin{cases}
 m_{j}<n_{j}\quad &\mathrm{if}\  j\in J\ \mathrm{and}\ j\in J^\prime, \\
 n_{j-1}\leq m_{j}\quad &\mathrm{if}\  j\in J\ \mathrm{and}\ j-1\in J^\prime.
\end{cases}$$
$$\bm{m}\inleq\bm{n}\xLeftrightarrow[]{\text{def}}\begin{cases}
 m_{j}\leq n_{j}\quad &\mathrm{if}\  j\in J\ \mathrm{and}\ j\in J^\prime, \\
 n_{j-1}< m_{j}\quad &\mathrm{if}\  j\in J\ \mathrm{and}\ j-1\in J^\prime.
\end{cases}$$
\end{dfn}
For a diagonally constant index $\bm{k}$ and $N>0$, Yamamoto (\cite{MSWyamamoto}) defined $\flat$-sum of Schur type as bellow:
$$\zeta_{<N}^{\flat}(\bm{k})=\sum_{\substack{\bm{n}_{p}^{(l)}\in[1,N-1]^{J_{p}}\\ \bm{n}_{p}^{(l)}\inleq \bm{n}_{p}^{(l+1)}(1\leq l<k_{p})\\ \bm{n}_{p}^{(k_{p})}\inle\bm{n}_{p+1}^{(1)}(p_{0}\leq p<p_{1})}}\prod_{p=p_{0}}^{p_{1}}\frac{1}{\prod(\bm{N}-\bm{n}_{p}^{(1)})\prod(\bm{n}_{p}^{(2)})\cdots\prod(\bm{n}_{p}^{(k_{p})})},$$
where, $\bm{N}-\bm{n}=(N-n_{j})_{j\in J}$. For non-negative integers $m,n$, a connector $C_{N}(m,n)$ is defined by
$$C_{N}(m,n)=\left. \binom{n}{m} \middle/ \binom{N-1}{m}. \right.$$
\begin{lem}[\cite{MSWyamamoto}, Lemma 3.8]\label{Lem:Yam3.8}
  $\mathrm{(1)}$ For $0<m<N$ and $0\leq n\leq n^\prime<N$, it holds
  \begin{align}
    \label{Lem:3.8(1)}\frac{1}{m}(C_{N}(m,n^\prime)-C_{N}(m,n))=\sum_{b=n+1}^{n^\prime}C_{N}(m,b)\frac{1}{b}.
    \end{align}
  $\mathrm{(2)}$ For $0\leq m\leq m^\prime$ and $0<n<N$, it holds
  \begin{align}
    \label{Lem:3.8(2)}\sum_{a=m+1}^{m^\prime}C_{N}(a,n)\frac{1}{n}=(C_{N}(m,n-1)-C_{N}(m^\prime,n-1))\frac{1}{N-n}.
    \end{align}
\end{lem}

\begin{dfn}[cf. \cite{MSWyamamoto}]\label{Yam:def}
  Let $(J,J^\prime)$ be a consecutive pair of intervals, and $\bm{m}=(m_{j})_{j\in J},\ \bm{n}=(n_{j})_{j\in J^\prime}$ be tuples of non-negative integers. We set $\widetilde{J}=J\cup J^\prime$ and define the symbols $\widetilde{m}_{j},\ \widetilde{n}_{j}$ for $j\in\widetilde{J}$ by 
  $$\widetilde{m}_{j}=\begin{cases}
    m_{j}\quad &\mathrm{if}\ j\in J,\\
    0\quad &\mathrm{if}\ j\notin J,
  \end{cases}\quad \widetilde{n}_{j}=\begin{cases}
    n_{j}\quad &\mathrm{if}\ j\in J^\prime,\\
    N-1\quad &\mathrm{if}\ j\notin J^\prime.
  \end{cases}$$
  Then we define
  $$D_{N}(\bm{m},\bm{n})=\det(C_{N}(\widetilde{m}_{j_{1}},\widetilde{n}_{j_{2}}))_{j_{1},j_{2}\in\widetilde{J}}\ .$$
\end{dfn}
For an interval $J$, we say a tuple $\bm{m}=(m_{j})_{j\in J}$ is \textit{non-decreasing} when $m_{j}\leq m_{j+1}$ holds for $j,\ j+1\in J$. 
\begin{lem}[\cite{MSWyamamoto}, Lemma 3.11]\label{Lem:Yam3.11}
  Let $(J,J^\prime)$ be a consecutive pair of intervals.\\
  $\mathrm{(1)}$ If $\bm{m}\in[1,N-1]^{J}$ and $\bm{n}\in[0,N-1]^{J^\prime}$ are non-decreasing, it holds
  $$\frac{1}{\prod(\bm{m})}D_{N}(\bm{m},\bm{n})=\sum_{\substack{\bm{b}\in[1,N-1]^{J}\\ \bm{b}\inleq \bm{n}}}D_{N}(\bm{m},\bm{b})\frac{1}{\prod(\bm{b})}.$$
  $\mathrm{(2)}$ If $\bm{m}\in[0,N-1]^{J}$ and $\bm{n}\in[1,N-1]^{J^\prime}$ are non-decreasing, it holds
  $$\sum_{\substack{\bm{a}\in[1,N-1]^{J^\prime}\\ \bm{m}\inle\bm{a}}}D_{N}(\bm{a},\bm{n})\frac{1}{\prod(\bm{n})}=D_{N}(\bm{m},\bm{n-1})\frac{1}{\prod(\bm{N-n})},$$
  where $(\bm{n}-\bm{1})=(n_{j}-1)_{j\in J^\prime}$.
\end{lem}
  Now we define a connected sum
\begin{alignat}{2}
  Z(\bm{k};a)=\sum_{\substack{\bm{m}_{p}(p_{0}\leq p \leq a)\\ \bm{n}_{p}^{(l)}(a<p\leq p_{1})}}\Bigg(\prod_{p=p_{0}}^{a}\frac{1}{\prod(\bm{m}_{p})^{k_{p}}}\Bigg)\cdot &D_{N}(\bm{m}_{a},\bm{n}_{a+1}^{(1)}\bm{-1})\\
  &\times\Bigg(\prod_{p=a+1}^{p_{1}}\frac{1}{\prod(\bm{N}\bm{-n}_{p}^{(1)})\prod(\bm{n}_{p}^{(2)})\cdots\prod(\bm{n}_{p}^{(k_{p})})}\Bigg).
  \end{alignat}
 Here $\bm{m}_{p}\in[1,N-1]^{J_{p}} (p_{0}\leq p\leq a)$ run satisfying $\bm{m}_{p_{0}}\inle\cdots\inle\bm{m}_{a}$, and $\bm{n}_{p}^{(l)}\in[1,N-1]^{J_{p}}(a<p\leq p_{1},\ 1\leq l\leq k_{p})$ run satisfying $\bm{n}_{p}^{(l)}\inleq\bm{n}_{p}^{(l+1)}(1\leq l<k_{p})$ and $\bm{n}_{p}^{(k_{p})}\inle\bm{n}_{p+1}^{(1)}(a<p<p_{1})$. By definition, boundary conditions are
$$Z(\bm{k};p_{1})=\zeta_{<N}(\bm{k}),\quad Z(\bm{k};p_{0}-1)=\zeta_{<N}^{\flat}(\bm{k}).$$
By applying \eqref{Lem:3.8(1)} $k_{a}$ times and \eqref{Lem:3.8(2)}, we get a transport relation
$$Z(\bm{k};a)=Z(\bm{k};a-1).$$
Then the \Cref{Thm:Yam3.6} holds as follows:
 \begin{alignat}{1}
 \zeta_{<N}(\bm{k})=Z(\bm{k};p_{1})=Z(\bm{k};p_{1}-1)=\cdots=Z(\bm{k};p_{0}-1)=\zeta_{<N}^{\flat}(\bm{k}). 
 \end{alignat} 

\subsection{Proof of \Cref{main thm 2}}
We aim to extend above \Cref{Lem:Yam3.8} and \Cref{Lem:Yam3.11} to their $q$-analogues. 

\begin{lem}\label{my lem 2}
   $\mathrm{(1)}$ For $0<m<N,$ and $0\leq n\leq n^\prime<N$, 
  \begin{align}
    \frac{q^m}{[m]_{q}}(C_{N}^{q}(m,n^\prime)-C_{N}^{q}(m,n))=\sum_{b=n+1}^{n^\prime}C_{N}^{q}(m,b)\frac{q^b}{[b]_{q}}. \label{3.5(1)} 
    \end{align}
  $\mathrm{(2)}$ For $0\leq m\leq m^\prime$ and $0<n<N$, 
  \begin{align}
    \sum_{a=m+1}^{m^\prime}C_{N}^{q}(a,n)\frac{q^{n-a}}{[n]_{q}}=(C_{N}^{q}(m,n-1)-C_{N}^{q}(m^\prime,n-1))\frac{1}{[N-n]_{q}}.\label{3.5(2)}
    \end{align}
  Here, $C_{N}^{q}(m,n)$ is
  $$C_{N}^{q}(m,n)=\left. \binom{n}{m}_{q} \middle/ \binom{N-1}{m}_{q}, \right.$$
  and $C_{N}^{q}(m,N-1)=1$, $C_{N}^{q}(m,n)=0$ if $m>n$.
\end{lem}

\begin{proof}
  First, it is sufficient to prove
  $$\frac{q^m}{[m]_{q}}\Big(C_{N}^{q}(m,b)-C_{N}^{q}(m,b-1)\Big)=C_{N}^{q}(m,b)\frac{q^b}{[b]_{q}}.$$
  If $m< b$, it is easily seen that
  \begin{alignat}{3}
    &\frac{q^m}{[m]_{q}}\Big(C_{N}^{q}(m,b)-C_{N}^{q}(m,b-1)\Big)\\
    &=\frac{q^m}{[m]_{q}}\Bigg(\frac{[b]_{q}!}{[m]_{q}![b-m]_{q}!}\cdot\binom{N-1}{m}_{q}^{-1}-\frac{[b-1]_{q}!}{[m]_{q}![b-m-1]_{q}!}\cdot\binom{N-1}{m}_{q}^{-1} \Bigg)\\
    &=\frac{q^m}{[m]_{q}}C_{N}^{q}(m,b)\Big(1-\frac{[b-m]_{q}}{[b]_{q}} \Big)\\
    &=C_{N}^{q}(m,b)\frac{q^b}{[b]_{q}}.
  \end{alignat}
 Note that the equality holds if $m\geq b$, since both sides are $0$. Second, it is sufficient to prove 
  $$\Big(C_{N}^{q}(a-1,n-1)-C_{N}^{q}(a,n-1) \Big)\frac{1}{[N-n]_{q}}=C_{N}^{q}(a,n)\frac{q^{n-a}}{[n]_{q}} .$$
  If $a<n-1$, it is easily seen that
  \begin{alignat}{4}
  &\Big(C_{N}^{q}(a-1,n-1)-C_{N}^{q}(a,n-1) \Big)\frac{1}{[N-n]_{q}}\\
  &=\Bigg(\frac{[n-1]_{q}!}{[a-1]_{q}![n-a]_{q}!}\cdot\frac{[a-1]_{q}![N-a]_{q}!}{[N-1]_{q}!}-\frac{[n-1]_{q}!}{[a]_{q}![n-a-1]_{q}!}\cdot\frac{[a]_{q}![N-a-1]_{q}!}{[N-1]_{q}!} \Bigg)\frac{1}{[N-n]_{q}}\\
  &=C_{N}^{q}(a,n)\Bigg(\frac{[a]_{q}}{[n]_{q}}\cdot\frac{[N-a]_{q}}{[a]_{q}}-\frac{[n-a]_{q}}{[n]_{q}} \Bigg)\frac{1}{[N-n]_{q}}\\
  &=C_{N}^{q}(a,n)\frac{q^{n-a}}{[n]_{q}}.
  \end{alignat}
  Note that the equality holds if $a\geq n-1$.
\end{proof}

Let $(J,J^\prime)$ be a consecutive pair of intervals, and $\bm{m}=(m_{j})_{j\in J},\ \bm{n}=(n_{j})_{j\in J^\prime}$ be tuples of non-negative integers. We define $\widetilde{J}=J\cup J^\prime$ and 
$$D_{N}^{q}(\bm{m},\bm{n})=\det\Big(C_{N}^{q}(\widetilde{m}_{j_{1}},\widetilde{n}_{j_{2}}) \Big)_{j_{1},j_{2}\in\widetilde{J}}\ ,$$
where $\widetilde{m}_{j}$ (resp. $\widetilde{n}_{j}$) means the same definition as \Cref{Yam:def}. We also define $[\bm{m}]_{q}=([m_{j}]_{q})_{j\in J}$.
\begin{lem}
  Let $(J,J^\prime)$ be a consecutive pair of intervals.\\
  $\mathrm{(1)}$ If $\bm{m}\in[1,N-1]^{J}$ and $\bm{n}\in[0,N-1]^{J^\prime}$ are non-decreasing, we have
  \begin{align}
  \prod\Bigg(\frac{q^{\bm{m}}}{\bm{[m]}_{q}}\Bigg)D_{N}^{q}(\bm{m},\bm{n})=\sum_{\substack{\bm{b}\in[1,N-1]^{J}\\ \bm{b}\inleq \bm{n}}}D_{N}^{q}(\bm{m},\bm{b})\prod\Bigg(\frac{q^{\bm{b}}}{\bm{[b]}_{q}}\Bigg). \label{lem1}
  \end{align}
  $\mathrm{(2)}$ If $\bm{m}\in[0,N-1]^{J}$ and $\bm{n}\in[1,N-1]^{J^\prime}$ are non-decreasing, we have
  \begin{align}
  \label{lem2}
    \sum_{\substack{\bm{a}\in[1,N-1]^{J^\prime}\\ \bm{m}\inle\bm{a}}}D_{N}^{q}(\bm{a},\bm{n})\prod\Bigg(\frac{q^{\bm{n}-\bm{a}}}{[\bm{n}]_{q}}\Bigg)=D_{N}^{q}(\bm{m},\bm{n-1})\frac{1}{\prod([\bm{N-n}]_{q})}.
\end{align}
\end{lem}

\begin{proof}
  By using \Cref{my lem 2}, it holds by the same method as Lemma 3.11 of \cite{MSWyamamoto}.
\end{proof}

For a diagonally constant index $\bm{k}$ and $a\in\{p_{0}-1,\ldots,p_{1}\}$, we define a connected sum
\begin{alignat}{2}
  Z^{q}(\bm{k};a)=\sum_{\substack{\bm{m}_{p}(p_{0}\leq p\leq a)\\ \bm{n}_{p}^{(l)}(a<p\leq p_{1})}}\Bigg(\prod_{p=p_{0}}^{a}\frac{\prod(q^{(k_{p}-1)\bm{m}_{p}})}{\prod([\bm{m}_{p}]_{q})^{k_{p}}}&\Bigg)\cdot D^{q}_{N}(\bm{m}_{a},\bm{n}_{a+1}^{(1)}\bm{-1})\\
&\times\Bigg(\prod_{p=a+1}^{p_{1}}\frac{\prod^\prime(q^{\bm{n}_{p}^{(1)}})\cdots\prod^\prime(q^{\bm{n}_{p}^{(k_{p})}})}{\prod([\bm{N}-\bm{n}_{p}^{(1)}]_{q})\prod([\bm{n}_{p}^{(2)}]_{q})\cdots\prod([\bm{n}_{p}^{(k_{p})}]_{q})}\Bigg),
  \end{alignat}
  where
          $$\sideset{}{'}{\prod}([\bm{m}]_{q})=\sideset{}{'}\prod(([m_{j}]_{q})_{j\in[j_{0},j_{1}]})=\prod_{j=j_{0}+1}^{j_{1}}[m_{j}]_{q}.$$
 Here $\bm{m}_{p}\in[1,N-1]^{J_{p}}(p_{0}\leq p\leq a)$ run satisfying $\bm{m}_{p_{0}}\inle\cdots\inle\bm{m}_{a}$, and $\bm{n}_{p}^{(l)}\in[1,N-1]^{J_{p}}(a<p\leq p_{1},\ 1\leq l\leq k_{p})$ run satisfying $\bm{n}_{p}^{(l)}\inleq\bm{n}_{p}^{(l+1)}(1\leq l<k_{p})$ and $\bm{n}_{p}^{(k_{p})}\inle\bm{n}_{p+1}^{(1)}$. By definition, boundary conditions are
 $$Z^{q}(\bm{k};p_{1})=\zeta_{<N}^{BZ}(\bm{k}),\quad Z^{q}(\bm{k};p_{0}-1)=\zeta_{<N}^{q\flat}(\bm{k}).$$
 Then we get the following:

 \begin{lem}[transport relation]\label{Lem:3.7}
   For any diagonally constant index $\bm{k}$ and $a\in\{p_{0},\ldots,p_{1}\}$, we have
   $$Z^{q}(\bm{k};a)=Z^{q}(\bm{k};a-1).$$
 \end{lem}

 \begin{proof}
   By using \eqref{lem1} $k_{a}$ times and \eqref{lem2}, the proof completes.
 \end{proof}

\begin{proof}[Proof of \Cref{main thm 2}]
  By \Cref{Lem:3.7}, we have
  $$\zeta_{<N}^{BZ}(\bm{k})=Z^{q}(\bm{k};p_{1})=\cdots=Z^{q}(\bm{k};p_{1}-1)=\cdots=Z^{q}(\bm{k};p_{0}-1)=\zeta_{<N}^{q\flat}(\bm{k}).$$
Thus we obtain the main result.
\end{proof}

 \subsection{Some examples}
 We give some examples about \Cref{main thm 2}.

 \noindent
\textbf{Case 1} : $\bm{k}=\begin{ytableau}
1 \\
2  
\end{ytableau}$. It is the same case as MSW formula. Each $J$ is $J_{2}=\emptyset,\ J_{1}=\{1\},\ J_{0}=\{1\},\ J_{-1}=\emptyset$. Now we start with
$$\zeta^{BZ}_{<N}(\bm{k})=Z^{q}(\bm{k};1)=\sum_{0<m_{0,1}<m_{1,2}<N}\frac{q^{m_{1,2}}}{[m_{0,1}]_{q}[m_{1,2}]_{q}^{2}}D_{N}^{q}(m_{1,2},N-1).$$
Note that $D_{N}^{q}(m_{1,2},N-1)=C_{N}^{q}(m_{1,2},N-1)$. By using \eqref{3.5(1)}, we obtain
$$\frac{q^{m_{1,2}}}{[m_{1,2}]_{q}}C_{N}^{q}(m_{1,2},N-1)=\sum_{0<b\leq N-1}C_{N}^{q}(m_{1,2},b)\frac{q^b}{[b]_{q}}.$$
Then we have
$$Z^{q}(\bm{k};1)=\sum_{\substack{0<m_{0,1}<m_{1,2}<N\\ 0<n<N}}\frac{1}{[m_{0,1}]_{q}[m_{1,2}]_{q}}C_{N}^{q}(m_{1,2},n)\frac{q^n}{[n]_{q}}.$$
By taking the same step, we get
$$\text{R.H.S.}=\sum_{\substack{0<m_{0,1}<m_{1,2}<N\\ 0<n_{1}\leq n_{2}<N}}\frac{1}{[m_{0,1}]_{q}}C_{N}^{q}(m_{1,2},n_{1})\frac{q^{n_{1}-m_{0,1}+n_{2}}}{[n_{1}]_{q}[n_{2}]_{q}}.$$
Next, we apply \eqref{3.5(2)} to obtain
$$\sum_{m_{0,1}<m_{1,2}\leq N-1}C_{N}^{q}(m_{1,2},n_{1})\frac{q^{n_{1}-m_{0,1}}}{[n_{1}]_{q}}=C_{N}^{q}(m_{0,1},n_{1}-1)\frac{1}{[N-n_{1}]_{q}},$$
then we get
$$\text{R.H.S.}=\sum_{\substack{0<m_{0,1}<N\\ 0<n_{1}\leq n_{2}<N}}\frac{1}{[m_{0,1}]_{q}}C_{N}^{q}(m_{0,1},n_{1}-1)\frac{q^{n_{2}}}{[N-n_{1}]_{q}[n_{2}]_{q}}=Z^{q}(\bm{k};0) .$$
We apply the same calculation, and finally, we obtain
$$\zeta^{BZ}_{<N}(\bm{k})=Z^{q}(\bm{k};1)=\cdots=Z^{q}(\bm{k};0-1)=\sum_{0<n_{1}<n_{2}\leq n_{3}<N}\frac{q^{n_{3}}}{[N-N_{1}]_{q}[N-n_{2}]_{q}[n_{3}]_{q}}=\zeta^{q\flat}_{<N}(\bm{k}).$$

\noindent
  \textbf{Case 2} : $\bm{k}=\begin{ytableau}
1 & 2  
\end{ytableau}$. Each $J$ is $J_{1}=\emptyset,\ J_{0}=\{1\},\ J_{-1}=\{-2\},\ J_{-2}=\emptyset$. Now we start with 
$$\zeta_{<N}^{BZ,\star}(\bm{k})=Z^{q}(\bm{k};0)=\sum_{0<m_{0,1}\leq m_{-1,2}<N}\frac{q^{m_{-1,2}}}{[m_{0,1}]_{q}[m_{-1,2}]_{q}^{2}}D_{N}^{q}(m_{0,1},N-1) .$$
Note that $D_{N}^{q}(m_{0,1},N-1)=C_{N}^{q}(m_{0,1},N-1)$. First, we apply \eqref{3.5(1)} to obtain
$$\frac{1}{[m_{0,1}]_{q}}C_{N}^{q}(m_{0,1},N-1)=\sum_{0<b\leq N-1}C_{N}^{q}(m_{0,1},b)\frac{q^{b-m_{0,1}}}{[b]_{q}},$$
then we have
$$Z^{q}(\bm{k};0)=\sum_{\substack{0<m_{0,1}\leq m_{-1,2}<N\\ 0<n<N}}\frac{q^{m_{-1,2}}}{[m_{-1,2}]_{q}^2}C_{N}^{q}(m_{0,1},n)\frac{q^{n-m_{0,1}}}{[n]_{q}}.$$
Next, we apply \eqref{3.5(2)} to obtain
  $$\sum_{0<m_{0,1}\leq m_{-1,2}}C_{N}^{q}(m_{0,1},n)\frac{q^{n-m_{0,1}}}{[n]_{q}}=\Big(1-C_{N}^{q}(m_{-1,2},n-1)\Big)\frac{1}{[N-n]_{q}},$$
then we have
  $$\text{R.H.S.}=\sum_{\substack{0<m_{-1,2}<N\\ 0<n<N}}\frac{q^{m_{-1,2}}}{[m_{-1,2}]_{q}^{2}}D_{N}^{q}(m_{-1,2},n-1)\frac{1}{[N-n]_{q}}=Z^{q}(\bm{k};-1).$$
  Here $D_{N}^{q}(m_{-1,2},n-1)=1-C_{N}^{q}(m_{-1,2},n-1)=C_{N}^{q}(m_{-1,2},N-1)-C_{N}^{q}(m_{-1,2},n-1)$. 
  We apply the same calculation, and finally, we obtain
  $$\zeta^{BZ,\star}_{<N}(\bm{k})=Z^{q}(\bm{k};0)=\cdots=Z^{q}(\bm{k};-1-1)=\sum_{\substack{n_{1}\leq n_{2}\geq n_{3}\\ {}^{\forall}n_{j}\in[1,N-1]}}\frac{q^{n_{2}}}{[N-n_{1}]_{q}[n_{2}]_{q}[N-n_{3}]_{q}}=\zeta_{<N}^{q\star\flat}(\bm{k}).$$
For example, if $N=3$ then
\begin{alignat}{2}
&\zeta_{<3}^{BZ,\star}(\bm{k})=q+\frac{q^2}{[2]_{q}^{2}}+\frac{q^{2}}{[2]_{q}^{3}},\\
&\zeta_{<3}^{q\flat}(\bm{k})=\frac{q}{[2]_{q}^{2}}+\frac{q^2}{[2]_{q}^{3}}+\frac{q^2}{[2]_{q}^{2}}+\frac{q^2}{[2]_{q}^{2}}+\frac{q^2}{[2]_{q}}.
\end{alignat}
Taking the difference, we obtain
$$\zeta_{<3}^{BZ,\star}(\bm{k})-\zeta_{<3}^{q\flat}(\bm{k})=q-\frac{q}{[2]_{q}^{2}}-\frac{q^2}{[2]_{q}^{2}}-\frac{q^2}{[2]_{q}}=q-(1+q)\frac{q}{1+q}=0.$$
\begin{rem}
   If $\bm{k}=\begin{ytableau}
k_{1} & \cdots & k_{r} 
\end{ytableau}\ (k_{1},\ldots,k_{r}\geq1)$, then we have
$$\zeta_{<N}^{BZ,\star}(\bm{k})=\zeta_{<N}^{q\star\flat}(\bm{k}),$$
where 
\begin{align}
  \zeta_{<N}^{BZ,\star}(\bm{k})&=\sum_{0<m_{1}\leq\cdots\leq m_{r}<N}\dfrac{q^{(k_{1}-1)m_{1}+\cdots+(k_{r}-1)m_{r}}}{[m_{1}]_{q}^{k_{1}}\cdots[m_{r}]_{q}
  ^{k_{r}}},\\
   \zeta_{<N}^{q\star\flat}(\bm{k})&= \sum_{\substack{0<n_{j1}\leq\cdots\leq n_{jk_{j}}<N(1\leq j\leq r)\\ n_{(j-1)1}\leq n_{jk_{j}}(1<j\leq r)}}\prod_{j=1}^r\frac{q^{n_{j2}+\cdots+n_{jk_{j}}}}{[N-n_{j1}]_{q}[n_{j2}]_{q}\cdots [n_{jk_{j}}]_{q}}.
\end{align} 
\end{rem}

\noindent
\textbf{Case 3} : $\bm{k}=\begin{ytableau}
 \none &  1\\
1 & 1
\end{ytableau}$. Each $J$ is $J_{2}=\emptyset,\ J_{1}=\{1\},\ J_{0}=\{2\},\ J_{-1}=\{2\},\ J_{-2}=\emptyset$. Now we start with
$$\zeta^{BZ}_{<N}(\bm{k})=Z^{q}(\bm{k};1)=\sum_{\substack{\hspace{6truemm} m_{-1,2}\\ \hspace{6truemm} \rotatebox{90}{>}\\ \hspace{-3truemm}m_{1,1}\hspace{1truemm}\leq \hspace{1truemm}m_{0,2}}}\frac{1}{[m_{-1,2}]_{q}[m_{0,2}]_{q}[m_{1,1}]_{q}}D_{N}^{q}(m_{1,1},N-1).$$
Note that $D_{N}^{q}(m_{1,1},N-1)=C_{N}^{q}(m_{1,1},N-1).$ First, we apply \eqref{3.5(1)} to obtain
$$\frac{1}{[m_{1,1}]_{q}}C_{N}^{q}(m_{1,1},N-1)=\sum_{0<b\leq N-1}C_{N}^{q}(m_{1,1},b)\frac{q^{b-m_{1,1}}}{[b]_{q}},$$
then we have
$$Z^{q}(\bm{k};1)=\sum_{\substack{m_{-1,2}<m_{0,2}\geq m_{1,1} \\ {}^{\forall}m_{i,j}\in[1,N-1]\\  0<n\leq N-1}}\frac{1}{[m_{-1,2}]_{q}[m_{0,2}]_{q}}C_{N}^{q}(m_{1,1},n)\frac{q^{n-m_{1,1}}}{[n]_{q}}.$$
Next, we apply \eqref{3.5(2)} to obtain
$$\sum_{0<m_{1,1}\leq m_{0,2}}C_{N}^{q}(m_{1,1},n)\frac{q^{n-m_{1,1}}}{[n]_{q}}=\Big(1-C_{N}^{q}(m_{0,2},n-1) \Big)\frac{1}{[N-n]_{q}},$$
then we have
$$\text{R.H.S.}=\sum_{\substack{0<m_{-1,2}<m_{0,2}<N\\ 0<n<N}}\frac{1}{[m_{-1,2}]_{q}[m_{0,2}]_{q}}D_{N}^{q}(m_{0,2},n-1)\frac{1}{[N-n]_{q}}=Z^{q}(\bm{k};0) .$$
Note that $D_{N}^{q}(m_{0,2},n-1)=1-C_{N}^{q}(m_{0,2},n-1)=C_{N}^{q}(m_{0,2},N-1)-C_{N}^{q}(m_{0,2},n-1)$. We apply the same calculation, and finally, we obtain
$$\zeta_{<N}^{BZ}(\bm{k})=Z^{q}(\bm{k};1)=\cdots=Z^q(\bm{k};0)=\cdots=Z^{q\flat}(\bm{k};-1-1)=\zeta_{<N}^{q\flat}(\bm{k}).$$

\section*{Acknowledgments}
The author expresses sincere thanks to Professor Yasuo Ohno for carefully reading and providing helpful comments. Additionally, the author would like to express gratitude to Professor Shin-ichiro Seki for his fruitful guidance. Also, appreciation is extended to Professor Shuji Yamamoto for valuable advice. Finally, the author wishes to thank his colleagues for their various comments and remarks.

\bibliographystyle{palin}
\bibliography{BibTeX}

\end{document}